\documentclass[10pt, reqno]{amsart}

\usepackage[skip=5pt plus1pt, indent=20pt]{parskip}

\usepackage{amsmath,amssymb,mathrsfs}
\usepackage{xcolor}
\colorlet{mdtRed}{red!50!black}
\definecolor{dblue}{rgb}{0,0,.6}
\usepackage[colorlinks,pagebackref=true]{hyperref}
\hypersetup{colorlinks,linkcolor={red},citecolor={blue},urlcolor={red}}
\usepackage[all]{xy}
\usepackage{tikz,tikz-cd,tkz-graph,enumerate}
\usepackage{extpfeil}

\hypersetup{colorlinks,linkcolor={blue},citecolor={blue},urlcolor={red}}
\renewcommand*{\backref}[1]{}
\renewcommand*{\backrefalt}[4]{[{%
		\ifcase #1 Not cited.%
		\or $\uparrow$~#2.%
		\else $\uparrow$~#2.%
		\fi%
	}]}

\usepackage{stmaryrd}

\newcommand{\mf}[1]{\mathfrak{#1}}
\newcommand{\mc}[1]{\mathcal{#1}}
\newcommand{\ms}[1]{\mathscr{#1}}
\newcommand{\bb}[1]{\mathbb{#1}}

\makeatletter
\def\@tocline#1#2#3#4#5#6#7{\relax
	\ifnum #1>\c@tocdepth 
	\else
	\par \addpenalty\@secpenalty\addvspace{#2}%
	\begingroup \hyphenpenalty\@M
	\@ifempty{#4}{%
		\@tempdima\csname r@tocindent\number#1\endcsname\relax
	}{%
		\@tempdima#4\relax
	}%
	\parindent\z@ \leftskip#3\relax \advance\leftskip\@tempdima\relax
	\rightskip\@pnumwidth plus4em \parfillskip-\@pnumwidth
	#5\leavevmode\hskip-\@tempdima
	\ifcase #1
	\or\or \hskip 2em \or \hskip 2homologyem \else \hskip 3em \fi%
	#6\nobreak\relax
	\dotfill\hbox to\@pnumwidth{\@tocpagenum{#7}}\par
	\nobreak
	\endgroup
	\fi}
\makeatother

\theoremstyle{plain}

\newtheorem{theorem}{Theorem}[section]

\newtheorem{lemma}[theorem]{Lemma}
\newtheorem{corollary}[theorem]{Corollary}
\newtheorem{proposition}[theorem]{Proposition}

\theoremstyle{definition}

\newtheorem{remark}[theorem]{Remark}

\newtheorem{definition}[theorem]{Definition}

\numberwithin{equation}{section}

\newcommand{\Pic}{{\rm Pic}}

\newcommand{\Spec}{{\rm Spec \,}}

\newcommand{\ie}{{\it i.e.\/},\ }

\renewcommand{\tilde}{\widetilde}

\newcommand{\A}{\mathbb A}

\newcommand{\simpp}{{\rm{\Delta^{op}Pshv}}}

\def\<{\langle}
\def\>{\rangle} 
\def\-{\overline} 
\def\~{\widetilde}
\def\^{\widehat}
\def\fr{\mathfrak}
\def\@{\mathcal}
\def\#{\mathbb}
\def\&{\mathbf}
\def\_{\underline}
 
\def\x{\times}

\input{xy}
\xyoption{all}
\usepackage{cleveref}

\usepackage{marvosym} 
\renewcommand{\email}[2][1]{\thanks{\textit{Email address}#1: \href{mailto:#2}{#2}}}

\renewcommand{\address}[2][1]{\thanks{\textit{Address}#1: #2}} 

\newcommand\fnnum[1]{\textsuperscript{#1}}
\makeatother

\begin{document}
	
	\title{$\mathbb{A}^1$-connectedness of moduli of vector bundles on a stacky curve}

	\author[S. Chakraborty]{Sujoy Chakraborty}
	\address[\fnnum{1}]{Department of Mathematics,
		Indian Institute of Science Education and Research Tirupati,
		Andhra Pradesh 517507, India}
	\email[\fnnum{1}]{sujoy.cmi@gmail.com, sujoychakraborty@labs.iisertirupati.ac.in}
	
	\author[S. Holme Choudhury]{Saurav Holme Choudhury}
	\address[\fnnum{2}]{Department of Mathematics,
		Indian Institute of Science Education and Research Tirupati,
		Andhra Pradesh 517507, India}
	\email[\fnnum{2}]{sourav.ac.93@gmail.com, sauravchoudhury@labs.iisertirupati.ac.in}
	
	\author[R. Pawar]{Rakesh Pawar}
	\address[\fnnum{3}]{School of Mathematical Sciences, National Institute of Science Education and Research,
		Bhubaneswar, Odisha- 752050, India.
	}
	\address[\fnnum{3}]{Homi Bhabha National Institute, Training School Complex, Anushakti Nagar, Mumbai 400094, India}
	\email[\fnnum{3}]{rakeshpawar@niser.ac.in}
	
	\begin{abstract} Let $\mf C$ be a projective stacky curve over an infinite field $k$ with the coarse space $C$, a geometrically connected smooth projective curve, together with a generating sheaf $\mc E$ on $\fr C$. We consider the moduli stack of $\mathcal{E}$-semistable vector bundles of fixed determinant and multiplicities on $\fr C$, and show that it is $\mathbb{A}^1$-connected if and only if it admits a $k$-rational point. We show that this holds whenever genus of $C$ is at least $2$ and $C(k)\neq \emptyset$. As a consequence, we conclude that, for any system of weights $\boldsymbol{\alpha}$, the moduli stack of parabolic $\boldsymbol{\alpha}$-semistable vector bundles of fixed determinant and quasi-parabolic data on a geometrically connected smooth projective curve is $\bb{A}^1$-connected. 
	\end{abstract}
	\maketitle
	\begin{footnotesize}
		\tableofcontents
	\end{footnotesize}
	
	\section{Introduction}
	A \textit{stacky curve} over a field $k$ is a smooth separated finite type geometrically connected  Deligne-Mumford stack $\mf C$ of dimension 1 such that there exists a non-empty scheme $X$ and an open immersion $X\to \fr C$, so it has generically trivial stabilizers.  
	The complement of this open subset consists of finitely many \textit{stacky points}, and the stacky nature of the curve $\fr C$ is encoded in the stabilizer group schemes at these stacky points. A \emph{tame stacky curve} $\fr C$ is a stacky curve such that the stabilizer group schemes at these stacky points are finite cyclic of order coprime to the characteristic of the base field $k$. This gives rise to new interesting features which distinguish a stacky curve from a classical (schematic) algebraic curve. For example, the fibre of a vector bundle (\ie a locally free sheaf of finite rank) $E\to \mf {C}$ at a stacky point $p$ with non-trivial stabilizer $\mu_e$ (where $\mu_e$ is the cyclic group scheme associated to the $e^{th}$ roots of unity) is a $\bb{Z}/e\bb{Z}$-graded vector space; here the abelian group $\bb{Z}/e\bb{Z}$ is thought of as the character group of the group scheme $\mu_e$. This means that beyond the rank and degree, there are additional numerical invariants for vector
	bundles on $\mf C$ called \textit{multiplicities}, which are the dimensions of the graded pieces of the fibre at	each stacky point $p$. 
	
	The underlying topological space of a stacky curve $\mf C$ admits the structure of a smooth algebraic curve $C$, known as the coarse space of $\mf C$. Unlike the case of classical algebraic curves where projectivity is defined using ample line bundles, the notion of projectivity of a stacky curve depends on the choice of certain \textit{generating sheaves}, which are higher-rank vector bundles with multiplicities of certain type \cite{OS03}. One of the motivations of studying vector bundles on a tame stacky curve $\mf C$ is their close relation to parabolic vector bundles on its coarse space $C$ when both $\mf C$ and $C$ are projective (see~\cite{Bor07, Nir08, DHMT24}). Parabolic vector bundles on a smooth projective classical curve over an algebraically closed field were introduced by Mehta and Seshadri \cite{MS80}; these are vector bundles on $C$ with prescribed flags in fibres over certain finitely many points on $C$ together with certain weights. Parabolic vector bundles and their moduli spaces and stacks play a prominent role in geometry. Thus, studying vector bundles on stacky curves and their moduli stacks gives an alternative approach to studying parabolic bundles and their moduli on classical curves. 
	
	Morel-Voevodsky \cite{MV99}, defined the $\mathbb{A}^1$-homotopy category that is suitable for studying algebraic geometric objects as spaces from the point of view of homotopy theory where the affine line $\#A^1$ is treated as the unit interval. In the vein of understanding the $\#A^1$-homotopy type of a space, the basic aspect one wishes to settle is whether the space is $\#A^1$-connected or not; in other words, determining the $0$-th $\mathbb{A}^1$-homotopy sheaf $\pi_0^{\mathbb{A}^1}$. In the current paper, we study the $\#A^1$-connectedness of particular moduli stacks of vector bundles on a curve or a stacky curve. Any stack defines an object in the $\bb{A}^1$-homotopy category (see~\Cref{sec:prelim-homotopy}) by regarding it as a simplicial presheaf via the nerve construction.  
	
	In~\cite{HY24}, Hogadi and Yadav established the $\bb{A}^1$--connectedness of the moduli stack $\mc{M}_C(n,L)$ of vector bundles  of rank $n$ and determinant $L$ on a smooth projective curve $C$ over an infinite field. In \cite[Theorem 2.16]{CC25}, the first and second-named authors showed that the open substack $\mc{M}^{ss}_C(n,L)\subset \mc{M}_C(n,L)$ of semistable vector bundles is also $\bb{A}^1$-connected. They also showed \cite[Theorem 4.3]{CC25} the $\bb{A}^1$-connectedness of the moduli stack $\mc{PM}_C^{\boldsymbol{e,m}}(n,L)$ of quasi-parabolic vector bundles on $C$ of fixed determinant $L$ and a
	quasi-parabolic data $(\boldsymbol{e,m})$ along a set of parabolic points
	on $C$; see \S~\ref{sec:parabolic}. However, the question of  
	$\bb{A}^1$-connectedness of the open substack $\mc{PM}_C^{\boldsymbol{\alpha}-ss,\boldsymbol{e,m}}(n,L)\subset \mc{PM}_C^{\boldsymbol{e,m}}(n,L)$ of $\boldsymbol{\alpha}$-semistable parabolic bundles 	 was shown for small and generic weights $\boldsymbol{\alpha}$ with $\gcd(n, \deg L) = 1$, and the general case remained unanswered.

	In this article, we address the general case of parabolic $\boldsymbol{\alpha}$-semistable bundles on a smooth projective curve for any weights $\boldsymbol{\alpha}$, by studying vector bundles on stacky curves. Recently, there has been progress made in understanding such vector bundles and their moduli spaces and stacks \cite{DHMT24, Taams}. More precisely, let $\mf C$ be a projective stacky curve over an infinite field $k$ with a projective coarse curve $C$ and coarse space map $\pi:\mf C \to C$. Fix a generating sheaf $\mc{E}$ on $\mf C$, a positive integer $n$, a line bundle $L$ on the coarse curve $C$, and multiplicities $\underline{m}$; see \S~\ref{vb on stacky curve}.  Let $$\mc{M}_{\mf{C}}\bigl(n,\underline{m},L\bigr)$$ denote the moduli stack of vector bundles on $\mf C$ of rank $n$, multiplicities $\underline{m}$ and satisfying $\det(\pi_*E)\cong L$. It contains the open substack $$\mc{M}_{\mf{C}}^{\mc{E}-\text{ss}}\bigl(n,\underline{m},L\bigr)$$ of $\mc{E}$-semistable vector bundles on $\mf C$. The main result of this article is the following.
	
	\begin{theorem}[\text{Theorem \ref{thm:main-A1sc}}]  Let $\mf{C}$ be a projective stacky curve over an infinite field $k$ and $\mc{E}$ be a generating sheaf on $\mf C$. Then $\mc{M}^{\mc{E}-\text{ss}}_{\mf{C}}\bigl(n,\underline{m},L\bigr)$ is $\bb{A}^1$-connected if and only if it admits a $k$-rational point.
	\end{theorem} 
	
	As a result, we deduce 
	\begin{corollary}[Corollary~\ref{prop:genus-atleast2}]
		Let $\mf{C}$ be a projective stacky curve over an infinite field $k$ and $\pi:\fr C\to C$ be the coarse space map. If the curve $C$ is of genus  $\geq 2$, and $C(k)\neq \emptyset$, then the moduli stack  $\mc{M}^{\mc{E}-\text{ss}}_{\mf{C}}\bigl(n,\underline{m},L\bigr)$ admits a $k$-rational point. Consequently, it is $\bb{A}^1$-connected.
		
	\end{corollary}

	\begin{corollary}[Corollary~\ref{cor:parabolic}] Let $C$ be a  geometrically connected smooth projective curve over an infinite field $k$ and $D$ be a finite set of closed points with residue fields separable over $k$. Let $\boldsymbol{\alpha}$ be a set of weights along $D$.  Then the moduli stack 
		$$\mc{PM}^{\boldsymbol{\alpha}-ss,\boldsymbol{e,m}}_C\left(n,L\right)$$ of $\boldsymbol{\alpha}$-semistable parabolic vector bundles on $C$ of rank $n$,
		determinant $L\in\Pic(C)$ and quasi-parabolic data $(\boldsymbol{e,m})$ $\#A^1$-connected if and only if it admits a $k$-rational point. 
		
		In particular, this holds for a geometrically connected smooth projective curve $C$ over an infinite field of genus at least $2$ and satisfying $C(k)\neq \emptyset$.
		
	\end{corollary}
	\noindent 
	\paragraph{\bf Structure of the paper:}
	
	In Section \ref{section:preliminaries}, we discuss some preliminary results on $\bb{A}^1$-connectedness in general. We also recall the notions of generating sheaves, slope and semistability for vector bundles on a stacky curve and establish some basic results.  In Section \ref{section:admissible type quotients}, we prove a  result regarding the existence of a surjective morphism between two vector bundles, where the image has multiplicities of an \textit{admissible} type. In Section \ref{section:a1-connectedness},  we establish the existence of direct $\bb{A}^1$-concordance (see Definition~\ref{def:concordance}) between two $K$-points of $\mc{M}_{\mf{C}}\bigl(n,\underline{m},L\bigr)$ for a field extension $K$ over $k$. This allows us to  establish a criterion for
	the $\bb{A}^1$-connectedness of the substack $\mc{M}_{\mf{C}}^{\mc{E}-\text{ss}}\bigl(n,\underline{m},L\bigr)$ in terms of existence of a $k$-rational point of the moduli stack. The result is proved using a Langton-type argument on (existence part of) the valuative criterion of properness for $\mc{E}$-semistable vector bundles on stacky curves \cite{Hua23}. In section~\ref{sec:parabolic}, via the isomorphism of moduli stacks $$\mc{PM}^{\boldsymbol{\alpha}-ss,\boldsymbol{e,m}}_C\left(n,L\right)\cong \mc{M}_{\mf{C}}^{\mc{E}_{\boldsymbol{\alpha}}-\text{ss}}\bigl(n,\underline{m},L\bigr)$$
	we deduce the $\bb{A}^1$-connectedness for the corresponding parabolic moduli stack on the coarse curve $C$, which generalizes the earlier result in \cite{CC25}. Section \ref{sec:rpt} then studies sufficient conditions on the stacky curve which ensures the existence of a $k$-rational point on the moduli stack.
    
	\subsection*{Acknowledgments}
	The first-named author is supported by the DST--INSPIRE Faculty \allowbreak Fellowship (Grant No.:~DST/ INSPIRE/04/2024/001521), Ministry of Science and Technology, Government of India. The second-named author is partially supported by the same grant. They both would also like to thank IISER Tirupati for excellent working conditions. The third-named author is supported by NISER, DAE, India. 
	
	\section{Preliminaries}\label{section:preliminaries}
	\subsection{\texorpdfstring{$\bb{A}^1$-connected components of simplicial presheaves}{simp}}\label{sec:prelim-homotopy}

	Let $S$ be a Noetherian scheme of finite Krull dimension, and let
	$\mathrm{Sm}_S$ denote the category of smooth schemes of finite type over
	$S$, equipped with the Nisnevich topology. The category of simplicial presheaves $Spc_S:=\simpp({\rm Sm}_S)$ is a simplicial model category with the Nisnevich-local injective model structure as in~\cite[Theorem 1.4]{MV99}. By \cite[section 3.2, page 105]{MV99}, the above simplicial model category is endowed with the $\A^1$-local model category structure and denoted by $Spc^{\A^1}_S.$ An object of $Spc^{\A^1}_S$ is referred as a \emph{space} over $S$. Let $\mathcal{H}(S)$
	denote the $\mathbb{A}^1$-homotopy category of spaces over $S$, which is the
	homotopy category obtained from the category $Spc^{\A^1}_S$ by inverting the
	$\mathbb{A}^1$-weak equivalences, as defined in \cite{MV99}.
	\begin{definition}[Sheaf of $\mathbb{A}^1$-connected components]
		\label{def:a1-connected-components}
		For a space
		$\mathcal{X} \in \mathcal{H}(S)$, the \textbf{sheaf of
			$\mathbb{A}^1$-connected components}, denoted by $\pi_0^{\mathbb{A}^1}(\mathcal{X})$,
		is the Nisnevich sheaf on $\mathrm{Sm}_S$ defined as the sheafification of
		the presheaf
		\[
		U \longmapsto \mathrm{Hom}_{\mathcal{H}(S)}(U, \mathcal{X}),
		\]
		where $U \in \mathrm{Sm}_S$. Equivalently, if
		$\mathrm{L}_{\mathbb{A}^1}: \mathrm{sPre}(\mathrm{Sm}_S) \to
		\mathrm{sPre}(\mathrm{Sm}_S)$ denotes the $\mathbb{A}^1$-localization functor,
		then
		\[
		\pi_0^{\mathbb{A}^1}(\mathcal{X}) = a_{\mathrm{Nis}}
		\!\bigg(U \mapsto \mathrm{Hom}_{\mathcal{H}(S)}(U, L_{\#A^1}\mathcal{X})\bigg),
		\]
		where $a_{\mathrm{Nis}}$ denotes the Nisnevich sheafification of the underlying presheaf. A space $\mathcal{X}$
		is called \textbf{$\mathbb{A}^1$-connected} if
		$\pi_0^{\mathbb{A}^1}(\mathcal{X}) \cong *$.
	\end{definition}
	\noindent
	Throughout the paper, we will take $S=\Spec k$ where $k$ is a field. 
	\begin{definition}[\text{\cite[Definition~2.6]{HY24}}]
		\label{def:naive-a1-homotopy}
		Let $\mf{X}$ be a simplicial presheaf on $\mathrm{Sm}_k$.
		
		\begin{enumerate}[(1)]
			\item Let $U\in \mathrm{Sm}_k$. Two objects $x$ and $y$ in $\mf{X}(U)$
			are said to be \textit{naively $\bb{A}^1$-homotopic} if there exists a
			map $f: \bb{A}^1\times_k U \rightarrow \mf{X}$ satisfying $f_0= x$ and
			$f_1 = y$, where $f_i$ is the composition
			$U=\{i\}\x U\xrightarrow{i}\bb{A}^1\times_k U \xrightarrow{f} \mf{X}$. 
			
			This generates an equivalence relation on the set $\mf{X}(U)_0$ of 0-simplices of the simplicial set $\mf{X}(U)$. Let us denote the equivalence relation by    $\sim$.
			\item We shall denote by $S^{\rm pre}(\mf{X})$ the presheaf on $\mathrm{Sm}_k$
			whose value on a $k$-scheme $U$ is given by
			\[
			U\mapsto \mf{X}(U)_0/{\sim},
			\]
			where $\sim$ is the equivalence relation in~(1). The {\bf sheaf of $\#A^1$-chain connected components}  of $\fr X$, denoted by $S(\mf{X})$ is the Nisnevich sheafification of the presheaf $S^{\rm  pre}(\mf{X})$.
		\end{enumerate}
	\end{definition}
	
	The following well-known lemma will be used crucially in determining the $\#A^1$-connectedness of a simplicial presheaf.
	\begin{lemma}(\cite[Lemma~6.1.3]{Mor05}, \cite[section 2, corollary 3.22]{MV99}) \label{lem:a1-connected-condition}
		Let  $\fr X$ be a space over an infinite
		field $k$. Then $\fr X$ is $\bb{A}^1$-connected if
		\[
		S^{\rm pre}(\fr X)(\mathrm{Spec}\,K) = \ast
		\]
		for every finitely generated separable field extension $K$ of $k$.
	\end{lemma} 
	\begin{proof}
		Assume that $S^{\rm pre}(\fr X)(\mathrm{Spec}\,K) = *$, hence $S(\fr X)(\mathrm{Spec}\,K) = *$,
		for every finitely generated separable field extension $K$ of $k$, since stalks are same for a presheaf and its sheafification. Note there exists a natural epimorphism $S(\fr X) \to \pi^{\bb{A}^1}_0(\fr X)$ of Nisnevich sheaves, hence induces a surjective map of stalks at points in the Nisnevich topology, in particular at $\Spec K$ for  every finitely generated separable field extension $K$ of $k$. Hence $\pi^{\bb{A}^1}_0(\fr X)(\Spec K)=\ast$. Since the base field $k$ is infinite, the lemma follows by \cite[Lemma~6.1.3]{Mor05}, \cite[section 2, corollary 3.22]{MV99}.
	\end{proof}
	
	\begin{remark}
		Note that the assumption in lemma \ref{lem:a1-connected-condition} necessarily requires $\fr X(K) \neq \emptyset$ for every finitely generated separable field extension $K$ of $k$.
	\end{remark}
	
	\begin{remark}
		(Stack as a simplicial sheaf) Any groupoid valued presheaf on ${\rm Sm}_S$ can be seen as a simplicial presheaf by applying the nerve construction section-wise. In particular, this applies to (algebraic) stacks, which allows us to ask questions regarding the $\mathbb{A}^1$-homotopy theoretic properties of the stacks. 
	\end{remark}

	\subsection{Vector bundles on stacky curves }\hfill\label{vb on stacky curve}
	
	A \textit{stacky curve} over a field $k$ is a smooth separated finite type geometrically connected Deligne-Mumford stack $\mf C$ of dimension 1 such that there exists a non-empty scheme $X$ and an open immersion $X\to \fr C$.
	A \emph{tame stacky curve} $\fr C$ is a stacky curve such that the stabilizer group schemes at these stacky points are finite cyclic of order coprime to the characteristic of the base field $k$.
	
	Let $\mf{C}$ be a projective stacky curve over a field $k$ \ie $\fr C$ is a tame stacky curve, with the coarse space map	$\pi : \mf{C} \to C$ where $C$ is a smooth projective curve over $k$ (see Definition \cite[Definition 1.3.1]{Taams}). We follow the notation and  conventions in \cite{Taams}. Throughout the paper, we will assume $\fr C$ a projective stacky curve over a field $k$.
	
	Let $\mathcal{F}$ be a vector bundle on $\mf{C}$. As $\mf{C}$ is tame, the stabilizer group scheme $G_p$ at any stacky point $p \in \mf{C}$ is cyclic of order $e_p$ \ie isomorphic to $\mu_{e_p}$. Let $i_p: \@G_p=B\mu_{e_p} \hookrightarrow \mf{C}$ be the inclusion of the residual gerbe at $p$. The restriction $i_p^*\mathcal{F}$ corresponds to a $\mathbb{Z}/e_p\mathbb{Z}$-graded $k(p)$-vector space, that is $i_p^*\mathcal{F} \simeq \oplus_{i=0}^{e_p - 1}k(i)^{\oplus m_{p, i}}$, where $k(i)$ is the unique irreducible $1$-dimensional $\mu_{e_p}$ representation of weight $i$. 
	
	\noindent
	The numbers $m_{p,i}$ are called the \textit{multiplicities} of $\mathcal{F}$ at $p$. The multiplicity vector of $\mathcal{F}$ at $p$ is the tuple
	$$m_p(\mathcal{F}) := (m_{p, 0}, m_{p, 1}, \dots, m_{p,e_p -1})$$
	
	\noindent
	and the collection of multiplicity vectors over the stacky points of $\mf{C}$ is denoted as 
	$$\underline{m}(\mathcal{F}) = (m_p(\mathcal{F}))_{p \in \fr D}$$
	
	\noindent
	where $\fr D$ denotes the set of stacky points of $\mf{C}$.

	\begin{definition}\label{def:type of a vector bundle on stacky curve}
		Let $\mathcal{E}$ be a vector bundle on $\mf{C}$ with the set of multiplicities $\_m(\@E)$ in the sense above. We define the \emph{weights} of $\mathcal{E}$ to be $w_{p, j}(\mathcal{E}) := \dfrac{\sum_{l = 1}^j m_{p,l}(\mathcal{E})}{\mathrm{rank }(\mathcal{E})}$, where $j$ runs from $0$ to $e_p - 1$.
	\end{definition}
	Given a vector bundle $\mathcal{E}$ on $\mf{C}$, we have the following invariants associated to it:
	
	\begin{itemize}
		\item $\textrm{rank}(\mathcal{E})$
		\item $\textrm{deg}(\pi_* \mathcal{E})$
		\item $m_{p}(\mathcal{E})$ are the multiplicity vectors at $p$ and the multiplicities of $\@E$ is denoted by the collection $\_m(\@E) =\{m_{p}(\@E): p \ \text{is a stacky point of} \ \mf C\}$. 
	\end{itemize}
	
	This tuple $\alpha(\mathcal{E}) = (\textrm{rank}(\mathcal{E}), \textrm{deg}(\pi_* \mathcal{E}), \underline{m}(\mathcal{E})) \in \mathbb{Z} \oplus \mathbb{Z} \oplus (\bigoplus_{p\in \fr D}\mathbb{Z}^{e_p})$ is called the \textit{type} of $\mc{E}$.

	\begin{remark}
		This tuple is the same as the numerical invariant of $\mathcal{E}$ which is a class in $K_0^{\textrm{num}}(\mf{C})$ as described in \cite[Section 1.2]{DHMT24}.
	\end{remark}
	\begin{remark}[Projection formula]\label{projform}
		For later use, we note that the projection formula holds for the coarse space morphism $\pi: \mf{C} \to C$. More precisely, for any quasi-coherent $\mathcal{O}_C$-module $G$ and any locally free $\mc{O}_{\mf{C}}$-module $\mc{F}$, we have a natural isomorphism
		$$\pi_*(\pi^*G \otimes_{\mf{C}} \mc{F}) \cong G\otimes_C \pi_*\mc{F}\,.$$
	\end{remark} 
	
	We now show that given any type $\alpha$, we can construct a vector bundle $E$ on $\mf{C}$ such that $\alpha(E) = \alpha$. In addition we can also specify $\det(\pi_*E)$.
	
	\begin{proposition}[Existence of vector bundle of prescribed type]\label{prop:exists-vb}
		Let $\mf{C}$ be a projective stacky curve over a field $k$. Fix $n \in \mathbb{N}$, $L \in \Pic(C)$ and a multiplicity vector $\underline{m}$. Then there exists a vector bundle $E$ on $\mf{C}$ defined over $k$ of rank $n$, $\det(\pi_*E) = L$ and  $\underline{m}(E) = \underline{m}$.       
	\end{proposition}
	
	\begin{proof}
		By \cite[Corollary 1.2.11]{Taams}, any invertible sheaf $\mc{L}$ on $\mf{C}$ can be written uniquely as $$\mc{L}\cong \pi^*M \otimes \bigotimes_p\mc{O}(\mc{G}_p)^{\otimes a_p}$$ where $M\in \Pic(C)$ and $\mc{G}_p$ is the residual gerbe at a stacky point $p$. This is defined over $k$.
		
		Given the multiplicity  vector $\underline{m}_p$  of a rank $n$, to produce a rank $n$ bundle on $\fr C$, at the stacky point $p$, we need to give the  $\mu_{e_p}$-representation on the fiber at $p$, so that the weight decomposition is given by
	
		$$\bigoplus^{e_p-1}_{i = 0} k(i)^{\oplus m_{p,i}}.$$
		Equivalently, this is given by the $n$-tuple
		$a_p = (a_{p,j})_{j=1}^n$
		such that $a_{p, j} \leq a_{p, j+1}$ and $0\leq a_{p,j} \leq e_p -1$, more explicitly, $$a_p=(\underbrace{0,\dots, 0}_{m_{p, 0}}, \underbrace{1,\dots, 1}_{m_{p, 1}},\dots, \underbrace{e_p-1,\dots, e_p-1}_{m_{p, e_p-1}}).$$ The weight $i$ appears in $a_p$ exactly $m_{p,i}$ times. 
		
		We define the line bundle $\mc{L}'_i :=  \bigotimes_p\mc{O}(\mc{G}_p)^{\otimes a_{p,i}}$. By construction, we have
		$m_{p, j}(\mc{L}'_i) = \delta_{j, a_{p, i}}.$
		Then we note that the vector bundle $$E' = \bigoplus_{i = 1}^n \mc{L}'_i$$ is of rank $n$ and has multiplicity
		$$m_{p,j}(E') = \sum_{i=1}^n m_{p,j}(\mc{L}'_i) = \sum_{i=1}^n \delta_{j, a_{p, i}} = m_{p, j}.$$
		Thus, $E'$ has the prescribed rank and  the multiplicities. The determinant $\det \pi_*E'$ maybe not be $L$, so we further modify $E'$ and consider  
		$$E = (\pi^*L\otimes\mc{L}'_1) \oplus \bigoplus_{i=2}^n \mc{L}'_i$$
		This gives $\pi_*E \simeq (L\otimes \pi_*\@L_1')\oplus(\bigoplus_{i=2}^n \pi_*\@L_i')\simeq L \oplus \mc{O}_C^{n-1}$, since $\pi_*(\@L'_j)=\@O_C$ for each $j$ and hence $\det(\pi_*E) \simeq L$.\end{proof}
	
	We now recall the notion of semistability for vector bundles on a projective stacky curve.  
	\begin{definition}[\text{\cite[Definition 1.3.13]{Taams}}]\hfill
		\begin{enumerate}[(1)]
			\item For a vector bundle $\mathcal{E}$ and a coherent sheaf $\mathcal{F}$ on $\mf{C}$, its \textit{$\mathcal{E}$-degree} is defined by
			\[
			d_{\mathcal{E}}(\mathcal{F}) \;=\; \deg \pi_* \mathcal{H}om(\mathcal{E}, \mathcal{F}) \;-\; \operatorname{rank} \mathcal{F} \cdot \deg \pi_* \mathcal{H}om(\mathcal{E}, \mathcal{O}_\mf{C}).
			\]
			The \textit{$\mathcal{E}$-slope} of $\mc F$ is defined as $\mu_\mathcal{E}(\mathcal{F}) = d_\mathcal{E}(\mathcal{F})/\operatorname{rank} \mathcal{F}$.
			\item $\mc {F}$ is said to be \textit{$\mc E$-semistable} (respectively, \textit{$\mc{E}$-stable}), if for any proper subsheaf $\mc{F}'\subset \mc{F}$, we have $\mu_{\mc{E}}(\mc{F}')\leq \mu_{\mc{E}}(\mc{F})$ (respectively, $\mu_{\mc{E}}(\mc{F}')< \mu_{\mc{E}}(\mc{F})$). 
		\end{enumerate}
	\end{definition}
	\noindent
	
	We recall the definition of generating sheaf from \cite{Taams}, which was originally defined in \cite{OS03}.
	
	\begin{definition}[Generating sheaf, {\cite[Def.~1.3.1, Thm.~1.3.13]{Taams}}]
		Let $\@E$ be a vector bundle on $\mf{C}$. We say $\@E$ is a \emph{generating sheaf} if, for every coherent sheaf $F$ on $\mf{C}$, the natural evaluation map
		\[
		\pi^*\pi_*\bigl(\mathcal{H}om(\@E, F)\bigr) \otimes \@E \longrightarrow F
		\]
		is surjective. Equivalently, $\@E$ is generating if and only if $m_{p,i}(\@E) \neq 0$ for every stacky point $p$ of order $e_p$ and every $i\in[0,e_p - 1]$.
	\end{definition}
	
	It is easy to see that, if $\mc{E}$ is a generating sheaf, then so are $\mc{E}^{\vee}$ and $\mc{E}\otimes\mc{L}$ for any line bundle $\mc{L}$ on $\mf{C}$.
	
	\noindent
	Fix a generating sheaf $\mc{E}$ on $\mf C$, a positive integer $n$, a line bundle $L$ on the coarse curve $C$, and multiplicities $\underline{m}$ as in \S~\ref{vb on stacky curve}. Let 
$$\mc{M}_{\mf{C}}\bigl(n,\underline{m},L\bigr)$$
\noindent
denote the moduli stack of vector bundles $E$ on $\mf C$ of rank $n$, multiplicities $\underline{m}$ and satisfying $\det(\pi_*E)\cong L$. It contains the open substack 
	$$\mc{M}_{\mf{C}}^{\mc{E}-\text{ss}}\bigl(n,\underline{m},L\bigr)$$
	of $\mc{E}$-semistable vector bundles $E$ on $\mf C$.

	\section{Existence of quotients of admissible type}\label{section:admissible type quotients}
	
	In this section, we prove some technical results which will be used to establish the $\bb{A}^1$-connectedness results in the later sections. To be more specific, given a vector bundle $E$ on $\mf C$ of type $\alpha$, we introduce the notion of \textit{admissible types} with respect to $\alpha$, and show that any vector bundle of admissible type can be written as a quotient of $E$ after a large enough twist.  
	
	Throughout the remainder of the paper, we fix the following notations. Let $\mf{C}$ be a tame projective stacky curve over a field $k$, with coarse space map	$\pi : \mf{C} \to C$ where $C$ is a smooth projective curve over $k$.

	Let $\mc{O}_C(1)$ be a very ample line bundle on $C$, and for any integer $n$, define $\mc{O}_{\mf{C}}(n):= \pi^*\mc{O}_{C}(n)$. Fix $E$ of type $\alpha$ in the sense of subsection~\ref{vb on stacky curve}. We say that a type $\beta$ is \textit{admissible} with respect to $\alpha$, if rank $\beta$ is $\textrm{rank } \alpha -1$, and for every stacky point $p$ the multiplicities of $\beta$ and $\alpha$ are the same for all but one weight $i_p$, that is, $m_{p, j}(\beta) = m_{p, j}(\alpha) - \delta_{i_p, j}$.
	
	\noindent
	Let $E$ be a vector bundle on $\mf{C}$ of type $\alpha$. Let $F$ be a vector bundle of type $\beta$, admissible with respect to $\alpha$. Let $F(m):=F\otimes \mc{O}_{\mf C}(m)$. Define
	\smallskip
	\begin{equation*}
		H_m := \mathcal{H}om(E, F(m))\,.
	\end{equation*}
	Let $p_i\in\mf{C}$ be a stacky point of order $e_i$. The fiber of $H_m$ at $p_i$ can be described as follows. \'Etale-locally around $p_i$, one has $\mf{C} = \bigl[\operatorname{Spec} R' / \mu_{e_i}\bigr]$, where $R'=\frac{R[t]}{(t^{e_i}-s)}$ and $\mu_{e_i}$ acts on $R'$ by $\zeta\cdot t = \zeta t$, and the coarse moduli  $C$ equals $\Spec R$, where $R=(R')^{\mu_{e_i}}$ is a DVR with uniformizer $s$. By \cite[Proposition 3.12]{Bor07}, the fibers of $E$ and $F(m)$ at $p_i$ have weight space decompositions under the $\mu_{e_i}$-action:
	\smallskip
	\begin{align*}
		E|_{p_i} = \bigoplus_{j=0}^{e_i-1} V_{i,j}\,\,\,\text{and}\,\,\,F(m)|_{p_i} = \bigoplus_{j=0}^{e_i-1}W_{i,j}
	\end{align*}
	where $V_{i,j}$ and $W_{i,j}$ are the $j$-weight spaces for the $\mu_{e_i}$ actions on $E|_{p_i}$ and $F(m)|_{p_i}$, respectively. Thus, the weight space decomposition of $(H_m)|_{p_i} = \text{Hom}_{\kappa(p_i)}(E|_{p_i}, F(m)|_{p_i})$ is given by
	\smallskip
	\begin{align*}
		(H_m)|_{p_i} = \text{Hom}_{\mu_{e_i}}\bigl(E|_{p_i},F(m)|_{p_i}\bigr)\oplus W\,,
	\end{align*}
	where $\text{Hom}_{\mu_{e_i}}\bigl(E|_{p_i},F(m)|_{p_i}\bigr) = (H_m|_{p_i})^{\mu_{e_i}}$ is precisely the $0$-weight space, and $W$ is consists of the remaining weight spaces with nonzero weights.
	\begin{proposition}\label{prop:gen}
		Let $E$ be a vector bundle of type $\alpha$ on $\mf C$. Consider a vector bundle $F$ on $\mf C$ of an admissible type $\beta$. Then there exists a positive integer $m(E,F)$ depending only on $E$ and $F$, such that for any $m\ge m(E,F)$, the sheaf $\pi_* H_m$ is globally generated on $C$. More precisely:
		\begin{enumerate}[(i)]
			\item for every stacky point $p_i$ of order $e_{p_i}$, let $x_i:= \pi(p_i)$; then the evaluation map
			\smallskip
			\[
			\operatorname{ev}_{p_i} : H^0(C, \pi_* H_m) \otimes k(x_i)\longrightarrow \bigoplus_{j=0}^{e_{p_i}-1} \operatorname{Hom}(V_{i,j}, W_{i,j}) = \operatorname{Hom}_{\mu_{e_i}}\!\big(E|_{p_i}, F(m)|_{p_i}\big)
			\]
			is surjective, and 
			\item for every non-stacky closed point $x$, the evaluation map 
			$$\operatorname{ev}_x : H^0(C, \pi_* H_m)\otimes k(x) \longrightarrow H_m|_x \cong \operatorname{Hom}\bigl(k(x)^n, k(x)^{n-1}\bigr)$$ is surjective. 
		\end{enumerate}
		
	\end{proposition}
	
	\begin{proof}

		We have $H_m\cong E^\vee\otimes F(m)$. By projection formula (see Remark \ref{projform}), it follows that
		\[\pi_*(H_m)= \pi_*\bigl(E^\vee\otimes F\otimes \pi^*\mc{O}_{C}(m)\bigr)\cong \pi_*\bigl(E^\vee \otimes F\bigr)\otimes \mc{O}_C(m)\,.\]
		Thus, there exists an integer $m(E, F)$ such that $\pi_*(H_m)$ is globally generated for $m\ge m(E,F)$.
		\noindent
		
		To prove $(i)$, let $p_i$ be stacky point on $\mf C$, and let $x_i=\pi(p_i)\in C$. Then the fiber $(\pi_*H_m)|_{x_i}$ also contains $\text{Hom}_{\mu_{e_i}}\bigl(E|_{p_i},F(m)|_{p_i}\bigr)$ as a direct summand. Moreover, the restriction of the natural map $\pi^*\pi_*H_m \longrightarrow H_m$ to the fiber at $p_i$ gives rise to the map
		\medskip
		\[(\pi_*H_m)|_{x_i} \longrightarrow (H_m)|_{p_i}\]
		whose image is exactly $\text{Hom}_{\mu_{e_i}}\bigl(E(m)|_{p_i},F|_{p_i}\bigr)=\bigoplus_{j=0}^{e_{p_i}-1} \operatorname{Hom}(V_{i,j}, W_{i,j})$. Pre-composing this map with the natural surjection $H^0(C,\pi_*H_m)\otimes k(x_i) \xtwoheadrightarrow{\,} (\pi_*H_m)|_{x_i}$ gives us part ($i$).

		\noindent
		Part ($ii$) is the classical statement for projective curves. 
	\end{proof}
	
	\begin{remark} 
		For $m\geq m(E,F)$, the vanishing of $H^1$ and the Riemann-Roch theorem for classical curves gives us a explicit dimension count of $\dim H^0(C, \pi_*H_m)$.
	\end{remark}
	\noindent
	Finally, we show the main result of this section, that a generic morphism $E \to F(m)$ in Proposition~\ref{prop:gen} is a surjective morphism of sheaves.
	
	\begin{proposition}[]\label{prop:surjection}
		Let $k$ be an infinite field. Fix the same hypotheses as in Proposition~\ref{prop:gen}. In addition, we assume $n = \textrm{rank }E\geq 2$. For $m \geq m(E,F)$, there exists a section in $H^0(\mf{C}, H_m)$ given by a surjective morphism $\phi : E \xtwoheadrightarrow{\,} F(m)$.
	\end{proposition}

	\begin{proof}
		Fix $m\geq m(E,F)$ as in Proposition~\ref{prop:gen}. Set $V= H^0(\mf{C}, H_m)=H^0(C, \pi_*H_m)$. Let $\fr D = \{p_1, \dots, p_r\}$ denote the stacky points in $\fr C$, and let $D :=\pi(\fr D)$. Given $\phi\in H^0(\mf{C}, H_m)$, we say $\phi$ satisfies the property
		\begin{center}
			$(\ast_p):$ if for point $p\in \fr C$, the induced map on fibers   $\phi|_p: E|_{p}\to F(m)|_{p}$ is not surjective.
		\end{center}
		We show that the subset of sections $\phi\in H^0(\mf{C}, H_m)$, in the affine space $\#A(V)$, which satisfies the property $(\ast_p)$, for some point $p\in \fr C$, is a proper subset of $\#A(V).$
		
		To see this, we form subsets \begin{enumerate}
			\item $Z\subset \#A(V)$ such that $\phi\in Z(k)$ iff ($\ast_p)$ holds, for some point $p\in \fr C-\fr D$, and 
			\item  for each stacky point $p_i\in \fr D$, subsets $Z_{p_i}\subset \#A(V)$ such that $\phi\in Z_{p_i}(k)$ iff ($\ast_{p_i})$ holds. 
		\end{enumerate}
		Finally, we show that $\overline{Z}\bigcup (\cup_{i=1}^r Z_{p_i})$ is a proper subset of $\#A(V).$
		
		\noindent
		We will do this in steps: 
		
		\noindent
		\underline{{\bf Step 1:} (for non-stacky points)} Given a vector bundle $\mc{F}$ on $C$, let us denote the geometric vector bundle over $C$ associated to $\mc{F}$ by $\#V(\mc{F})=\Spec (S(\mc{F}^{\vee}))$, the spectrum of the symmetric $\@O_{C}$-algebra $S(\mc{F}^{\vee})$ associated to $\mc{F}^{\vee}$ (see~\cite[Chapter 2, Exercise 5.18]{Har} for notation). Let $\mc{V}:= V\otimes \mc{O}_C$, and consider the evaluation morphism
		\[
		\operatorname{ev} : \mc{V}\longrightarrow \pi_*H_m=\pi_*\operatorname{\@Hom}_{\@O_{C}}\!\big(E, F(m)\big)\] 
		which is shown to be surjective in Proposition~\ref{prop:gen}. The induced morphism of geometric vector bundles $\#V(\mc{V})\to \#V(\pi_*H_m) $ is epimorphism of schemes. Hence the restriction to $(C-D) \cong (\fr C-\fr D)$ is also surjective. Let $Y$ be the closed subscheme of   $\#V(\pi_*H_m) $ associated to the  ideal sheaf $\@I$ of the $\@O_{C}$-algebra $S((\pi_*H_m)^\vee)$ determined on a trivializing Zariski open cover by the ideal generated by determinants of $(n-1)\x(n-1)$ minors on each open set in the cover. Note that $Y$ is of codimension at least 2 in $\#V(\pi_*H_m) $, since it is codimension at least 2 when restricted to the trivializing Zariski cover.   
		
		\noindent
		Consider the fiber product $Z':=Y\x_{\#V(\pi_*H_m)}\#V(\mc V)$. Since the morphism $\#V(\mc V)\to \#V(\pi_*H_m) $ is flat, the inverse image of $Y$ (which is $Z'$) is of codimension at least $2$ in $\#V(\mc{V})$. Now consider the projection $p: \#V(\mc V)|_{C \setminus D} = (C \setminus D)\times \#A(V) \to \#A(V)$. Define $Z  := p(Z')$. Then, $$\dim Z \leq \dim Z' \leq  \dim \#V(\mc V)-2=\dim V-1.$$ Thus $Z$ is a proper constructible subset of $\#A(V)$. Clearly, $Z$ is the required subset as in (1).
		
		\noindent
		\underline{{\bf Step 2:} (for stacky points)}  Let $p_i\in \fr D$ be a stacky point. By Proposition~\ref{prop:gen}, $$\operatorname{ev}_{\pi(p_i)} : V \twoheadrightarrow \bigoplus_j \operatorname{Hom}_{\mu_{p_i}}(V_{i,j}, W_{i,j})$$ is surjective linear map. In the target,  there exists a unique index $j_i$ such that $\dim W_{i,j_i} = \dim V_{i,j_i}-1$, and $\dim W_{i,j} = \dim V_{i,j}$ for $j\neq j_i$. The non-surjective locus is the union over $j$ of the loci where the $j$-th block fails to surject: 
		
		\noindent
		for $j \neq j_i$ the non-surjective locus in $\operatorname{Hom}(V_{i,j}, W_{i,j})$ is a determinantal hypersurface (codimension $1$); 
		
		\noindent
		for $j = j_i$ the non-surjective locus in $\operatorname{Hom}(V_{i,j}, W_{i,j})$ has codimension $2$. 
		
		\noindent
		In all cases, the union is a proper closed subset of codimension at least 1, so its preimage  $Z_{p_i} \subset \mathbb{A}(V)$ under the map $ev_{\pi(p_i)}$ is a closed subscheme of dimension at most $(\dim V - 1)$. As the stacky locus is $0$-dimensional, so, unlike the classical count, codimension $1$ suffices here.
		
		\noindent
		\underline{{\bf Step 3:}} 
		Thus, from Steps 1 and 2, the locus of non-surjective $\phi$ is contained in $\overline{Z} \bigcup (\cup_{i=1}^r Z_{p_i})$, which is of dimension at most $\dim V - 1 < \dim V = \dim \mathbb{A}(V)$ hence $\overline{Z} \bigcup (\cup_{i=1}^r Z_{p_i})$ is a proper closed subscheme of $\#A(V)$. Since by assumption, $k$ is an infinite field, there is a $k$-rational point $\phi$ in the non-empty open complement of  $\overline{Z} \bigcup (\cup_{i=1}^r Z_{p_i})$ in $\#A(V)$ \ie $\phi$ is surjective on every fiber.
	\end{proof}
	\begin{remark}\label{rem:generic surjection}
		The above result holds for any $E$ and $F$ such that $F$ has admissible type with respect to $E$ and $\pi_*\mathcal{H}om(E, F)$ is globally generated.    
	\end{remark}
	
	\begin{remark}
		For two types $\alpha$ and $\beta$, there is a partial order where $\alpha \geq \beta$ if $m_{p, i}(\alpha) \geq m_{p,i}(\beta)$ for each stacky point $p$ and each weight $i \in \mathbb{Z}/e_p\mathbb{Z}$. Then using the above results iteratively, one gets that for vector bundles $E$ of type $\alpha$ and $F$ of type $\beta$ on $\mf{C}$ with $\alpha \geq \beta$, for large enough $m$, we get a surjection $E \to F(m)$.
	\end{remark}

	\section{\texorpdfstring{Characterization for $\bb{A}^1$-connectedness of $\mc{M}^{\mc{E}-\text{ss}}_{\mf{C}}\bigl(n,\underline{m},L\bigr)$}{a1-concordance-ss}}\label{section:a1-connectedness}

	We begin by defining the notion of $\mathbb{A}^1$-concordance for vector bundles on a stacky curve, directly analogous with the definition for schemes in \text{\cite[Definition 5.2]{AKW17}}.

	\begin{definition}\label{def:concordance}
		Let $\mf{C}$ be a stacky curve over a field $K$. Choose two distinct points $x_0$ and $x_1$ in $\bb{A}^1_K$. Two vector bundles $\mc{E}_0$ and $\mc{E}_1$ on $\mf{C}$ are said to be \textit{directly $\bb{A}^1$--concordant} if there exists a vector bundle $\mc{E}$ over $\mf{C}\times \bb{A}^1_K$ such that $\mc{E}_0\xrightarrow{\,\,\simeq\,\,}\mc{E}|_{\mf{C}\times\{x_0\}}$ and $\mc{E}_1\xrightarrow{\,\,\simeq\,\,}\mc{E}|_{\mf{C}\times\{x_1\}}$. The vector bundles $\mc{E}_0$ and $\mc{E}_1$ are said to be \textit{$\bb{A}^1$--concordant} if they are equivalent under the equivalence relation generated by direct $\bb{A}^1$--concordance. 
	\end{definition}
	
	\begin{remark}
		It is immediate that the equivalence relation generated by direct $\mathbb{A}^1$-concordances is the same as the equivalence relation defined in Definition \ref{def:naive-a1-homotopy}, when $\fr{X}$ is the moduli stack of vector bundles on $\mf{C}$.
	\end{remark}

	\begin{proposition}\label{prop:concordance}
		Let $\mf{C}$ be a projective stacky curve over an infinite field $k$ with coarse space $C$. Consider a field extension $K\supset k$, and let $E_0$ and $E_1$ be two vector bundles on $\fr C_K$ of rank $n$
		and multiplicities $\boldsymbol{m}$ satisfying $\det(\pi'_*E_i)\simeq \mc L$ for some line bundle $\mc L\in \Pic(C_K)$, where $\pi':\mf C_K \to C_K$ is the coarse moduli map. Then $E_0$ and $E_1$ are directly $\#A^1$-concordant.
	\end{proposition}
	\begin{proof}
		For a vector bundle $F$ on $\mf C_K$ and an integer $m\ge 0$, let $F(m):= F\otimes \pi'^*\mc{O}_{C_K}(m)$ as before. For $m$ large enough, Proposition \ref{prop:surjection} provides surjections $E_i\xtwoheadrightarrow{\,} F(m)$ for $i=0,1$. This gives rise to two short exact sequences
		\[0\longrightarrow K_i \longrightarrow E_i\longrightarrow F(m)\longrightarrow 0,\,\,i=0,1\]
		where $K_i$ are line bundles. We claim that $K_0\cong K_1$. To see this, we first apply $\pi'_*$ to the exact sequences above. Since $\pi'_*$ is an exact functor on the category of coherent sheaves on $\fr C$ (as $\fr C$ is tame), this gives 
		\[\pi'_*(K_i)\cong \det(\pi'_*E_i)\otimes \det(\pi'_*F(m))^{-1} \cong \mc{L}\otimes \det(\pi'_*F(m))^{-1},\,\,i=0,1.\]
		Moreover, the restriction of these short exact sequences over a stacky point $p$ of order $e$, give rise to short exact sequences of $\mu_{e}$-representations. Since $F$ and $F(m)$ have the same type, this immediately implies that the multiplicity vectors $m_p(K_i)$ of $K_i$ over $p$ are the same. Suppose $m_p(K_i)$ has 1 at position $a_p$, and 0 elsewhere. Then, we get
		\[K_i\cong \pi'^*(\pi'_*K_i)\otimes \otimes_p \mc{O}(\mc{G}_p)^{\otimes a_p},\,\,i=0,1,\]
		which implies $K_0\cong K_1$, as claimed. 
		
		Let $K_i\cong M$ for some $M\in \Pic(\mf C_K)$. Thus, $E_0$ and $E_1$ give rise to two extension classes $e_0$ and $e_1$ in $\text{Ext}^1(F(m),M)$. Using the results from base change and cohomology for tame stacks in the case $\mf{C}_K \longrightarrow
		\Spec K$ (see \cite[Remark A.2.5]{Bro12}) and in analogy with \cite{Lan83}, it follows that the moduli functor $\mc X : {\rm Sch}/K \longrightarrow \rm{Sets}$ which sends 
		$$ S\mapsto \mc{X}(S)=H^0\Bigl(S, \&{Ext}^1_{p_S}\bigl(q_S^*F(m), q_S^*M\bigr)\Bigr),$$ 
		where $q_S$ is as in the diagram below:
		$$\xymatrix{
			S\x \mf C_{K} \ar[r]^{q_S} \ar[d]^{p_S} & \mf C_{K} \ar[d]^f \\
			S\ar[r] \ar[r] & \Spec K}
		$$ 
		is representable by an affine space $V=\bb{A}^N_K$. The extension classes $e_0$ and $e_1$ can be considered as points in $V$. Consider the line joining $e_0$ and $e_1$; this gives rise to a map $f:\bb{A}^1_K \longrightarrow V$ satisfying $f(0)=e_0$ and $f(1)=e_1$. Pulling back the universal extension on $\mf{C}_K\times V$ to $\mf{C}_K\times \bb{A}^1_K$ gives rise to a direct $\bb{A}^1$-concordance between $E_0$ and $E_1$.\end{proof}

	\begin{remark}
		In the above proof, we use the existence of a vector bundle $F$ of rank $(n-1)$ having multiplicities as in Proposition \ref{prop:surjection}. This is ensured by Proposition~\ref{prop:exists-vb}. See also \cite[Corollary 1.2.11]{Taams} when $n = 2$.   
	\end{remark}

	We are now in a position to establish our main result concerning the $\bb{A}^1$-connectedness of the moduli stack $\mc{M}^{\mc{E}-\text{ss}}_{\mf{C}}\bigl(n,\underline{m},L\bigr)$; see \S~\ref{vb on stacky curve} for notations.
	
	\begin{remark}
		For $n=1$, it is easy to see that $\mc{M}^{\mc{E}-\text{ss}}\bigl(1,\underline{m},L\bigr)$ is $\#A^1$-connected. For this, note that we have $\mc{M}^{\mc{E}-\text{ss}}\bigl(1,\underline{m},L\bigr)=\mc{M}\bigl(1,\underline{m},L\bigr)\simeq B\bb{G}_m$ by \cite[Corollary 1.2.11]{Taams}, which is clearly $\#A^1$-connected. 
	\end{remark}
	
	\begin{theorem}\label{thm:main-A1sc}
		Let $\mf{C}$ be a projective stacky curve over an infinite field $k$ and $\mc{E}$ be a generating sheaf on $\mf C$. Let $n\geq 2$. Then $\mc{M}^{\mc{E}-\text{ss}}_{\mf{C}}\bigl(n,\underline{m},L\bigr)$ is $\bb{A}^1$-connected if and only if it admits a $k$-rational point.
	\end{theorem}
	
	\begin{proof}
		Assume $\mc{M}^{\mc{E}-\text{ss}}_{\mf{C}}\bigl(n,\underline{m},L\bigr)$ is $\bb{A}^1$-connected. Then by \cite[Remark 2.5, page 106]{MV99}, the moduli stack  $\mc{M}^{\mc{E}-\text{ss}}_{\mf{C}}\bigl(n,\underline{m},L\bigr)$ admits a $K$-rational point for any choice of $\underline{m}$ and $L$, and for any finitely generated field extension $K$ over $k$ (for this we note that for such $K$, $\Spec K$ is a Henselian local scheme).
		
		For the converse, we note that as $\mc{M}^{\mc{E}-\text{ss}}_{\mf{C}}\bigl(n,\underline{m},L\bigr)(k) \neq \emptyset$, then $\mc{M}^{\mc{E}-\text{ss}}_{\mf{C}}\bigl(n,\underline{m},L\bigr)(K) \neq \emptyset$ for any finitely generated field extension $K$ over $k$. 
		
		Let $n\ge 2$. Let $E_0$ and $E_1$ be two $\mc E_K$-semistable vector bundles on $\mf C_K$ considered as the $K$-points of $\mc{M}^{\mc{E}-\text{ss}}_{\mf{C}}\bigl(n,\underline{m},L\bigr)$, which is nonempty by given assumptions. By Proposition~\ref{prop:concordance}, we obtain a direct $\bb{A}^1$-concordance between $E_0$ and $E_1$. This implies the existence of a vector bundle $\ms E$ on $\mf C_K\times_K \bb{A}^1_K$ satisfying $\ms E\big|_{\mf C_K\times\{i\}}\simeq E_i,\,\,i=0,1$. As $\mf C_K\times_K \bb{A}^1_K\simeq \mf C\times_k \bb{A}^1_K$, $\ms E$ corresponds to a morphism
		\begin{align*}
			\varphi: \bb{A}^1_K \longrightarrow \mc{M}_{\mf{C}}\bigl(n,\underline{m},L\bigr)
		\end{align*}
		satisfying $\varphi(0)=E_0$ and $\varphi(1)=E_1$. Now, the semistable locus of the family $\ms E$, namely $$U:= \Bigl\{p\in \bb{A}^1_K\,\mid\,\ms E\big|_{\mf C_K\times\{p\}}\,\,\text{is}\,\,\mc{E}_K-\text{semistable}\Bigr\}$$ 
		is open in $\bb{A}^1_K$. Thus, $\varphi\big|_U$ factors through the open substack $\mc{M}_{\mf{C}}^{\mc E-ss}\bigl(n,\underline{m},L\bigr)$, such that the diagram
		\begin{align*}
			\xymatrix{
				\bb{A}^1_K \ar[rr]^(.45){\varphi} && \mc{M}_{\mf{C}}\bigl(n,\underline{m},L\bigr) \\
				U \ar@{^{(}->}[u] \ar[rr]^(.45){\varphi|_{U}} && \mc{M}_{\fr C}^{\mc E-ss}\bigl(n,\underline{m},L\bigr)
				\ar@{^{(}->}[u]
			}
		\end{align*} commutes.
		Since the existance part of the valuative criterion for $\mc{M}_{\mf{C}}^{\mc E-ss}\bigl(n,\underline{m},L\bigr)$ holds (see~\cite[Theorem 1.1]{Hua23}), the map $\varphi\big|_{U}$ can be extended to a morphism $\tilde{\varphi}: \bb{A}^1_K\longrightarrow \mc{M}_{\mf{C}}^{\mc E-ss}\bigl(n,\underline{m},L\bigr)$. Note that $\{0,1\}\subset \mc{U}$ by assumption, and thus their images remain unchanged under the modification $\tilde{\varphi}$:
		\[\tilde{\varphi}(i)=\varphi(i)=E_i,\,\,\forall\,\,i=0,1\,.\]
		This gives a direct $\bb{A}^1$-concordance between $E_0$ and $E_1$ lying inside $\mc{M}_{\mf{C}}^{\mc E-ss}\bigl(n,\underline{m},L\bigr)$. It follows from Lemma \ref{lem:a1-connected-condition} that $\mc{M}_{\mf{C}}^{\mc E-ss}\bigl(n,\underline{m},L\bigr)$ is $\bb{A}^1$-connected.
	\end{proof}
	
	\begin{remark}
		The existance part of valuative criterion for $\@X=\mc{M}_{\mf{C}}^{\mc E-ss}\bigl(n,\underline{m},L\bigr)$ as established in~\cite[Theorem 1.1]{Hua23}, shows that given a DVR $R$ over $k$ with the residue field $k$ and its fraction field $L$ and an $\@E$-semistable bundle on $\fr C\x_k \Spec L$ of rank $n$ and  multiplicity $\_m$, one can extend it to an $\@E$-semistable bundle on $\fr C\x_k \Spec R$. In particular one doesn't need to consider a field extension in the valuative criterion.           
		Although \cite{Hua23} works with the assumption of an algebraically closed base field, universally closedness is fpqc local for morphism of stacks $\mathcal{X} \to \textrm{Spec}(k)$. Also the existence part of the valuative criterion is equivalent to the universally closedness (see~\cite[\href{https://stacks.math.columbia.edu/tag/0CLW}{0CLW} and \href{https://stacks.math.columbia.edu/tag/0CLX}{0CLX}]{SP}). 
	\end{remark}
	
	\begin{remark}
		The construction in the above proposition implies that the  full moduli stack $\mc{M}_{\mf{C}}\bigl(n,\underline{m},L\bigr)$ is $\mathbb{A}^1$-connected under the given assumptions as the existence of a $k$-rational point is given by proposition \ref{prop:exists-vb}. This was already shown in \cite[Theorem 3.7]{CC25} by viewing it as an iterated flag bundle over $\mathcal{M}_C(n, L)$. The above result gives an alternate proof by constructing direct $\mathbb{A}^1$-concordances in $\mc{M}_{\mf{C}}\bigl(n,\underline{m},L\bigr)$. 
	\end{remark}

	\section{\texorpdfstring{$\bb{A}^1$-connectedness of moduli stack of parabolic bundles}{a1-connected-pb}}\label{sec:parabolic}
	
	Fix a finite subset
	$D\subset C$ of distinct closed points on the curve $C$, called \textit{parabolic points},
	and the following data:
	\begin{enumerate}[$\bullet$]
		\item a positive integer $n$,
		\item for each $p\in D$,  a tuple of non-negative
		integers $\underline{m}(p):=(m_{p,0}\,,m_{p,1},\ldots,\allowbreak  m_{p,e_p-1})$ and a positive integer $e_p$
		satisfying $\sum\nolimits_{i=0}^{e_p-1}m_{p,i} = n$.
	\end{enumerate}
	Set $\boldsymbol{e} :=\{e_p\mid p\in D\}$ and
	$\boldsymbol{m} := \{\underline{m}(p)\mid p\in D\}$. The tuple $(\boldsymbol{e,m})$ is called a \textit{quasi-parabolic data of
		rank $n$ along $D$}. The collection $\boldsymbol{e}$ is a system of
	\textit{lengths} along $D$, and $\boldsymbol{m}$ is a system of
	\textit{multiplicities} along $D$.
	
	\noindent
	In what follows, we shall always assume that for each $p\in D$, the residue field $k(p)$ is separable over $k$. In fact, by \cite[Remark 1.1.34]{Taams}, this holds whenever $C$ is the coarse space of a tame stacky curve $\mf C$.
	\begin{definition}\label{def:quasi-parabolic-bundle}
		Fix a quasi-parabolic data $(\boldsymbol{e,m})$. A \textit{quasi-parabolic
			vector bundle of rank $n$} on $C$ is a rank $n$ vector bundle $E$ on $C$
		together with, for each $p\in D$, a flag of type $\underline{m}(p)$ on the
		fiber $E_p$:
		\smallskip
		\[
		E_p = E_{p,0}\supseteq E_{p,1}\supseteq \cdots
		\supseteq E_{p,e_p-1} \supseteq E_{p,e_p} = 0,
		\]
		satisfying $\dim(E_{p,i}/E_{p,i+1})=m_{p,i}$ for all $0\leq i< e_p$.
	\end{definition}
	\begin{remark}
		Our definition of quasi-parabolic bundle differs from the original definition of Mehta-Seshadri \cite{MS80}, in the sense that the flags at the parabolic points are not required to be strict. This is the convention followed by Maruyama-Yokogawa; see \cite[Definition 1.3 and Remark 1.1]{Yok95}, where non-strict flags are allowed in the definition of parabolic bundles. Here, we need to work with non-strict flags for several reasons. First of all, given a vector bundle $\mc F$ on $\mf C$, the quasi- parabolic structure on $\pi_*\mc F$ via Borne's correspondence \cite{Bor07} may result in repetitions in a flag at a parabolic point; see also \cite[Definition 1.4.3]{Taams}. Secondly, this convention will be useful in our later discussion, when we  identify the moduli stack of vector bundles on $\mf C$ of a given type with the moduli stack of quasi-parabolic bundles on $C$ for a suitably chosen quasi-parabolic data in \eqref{eq:parabolic-stacky}.
	\end{remark}
	For a vector bundle $\mc{F}$ of rank $n$ on a $k$-scheme $S$ and a tuple
	$\underline{m}=(m_0,\ldots,m_{e-1})$ with $\sum m_i=n$, let
	$\mathrm{Flag}_{\underline{m}}(\mc{F})\longrightarrow S$ denote the
	associated flag bundle of type $\underline{m}$. By its universal property, a
	section of this flag bundle corresponds to a chain of sub-bundles of $\mc{F}$
	of type $\underline{m}$.
	
	\begin{definition}[\text{\cite[Definition~3.1]{BY99}}]
		\label{def:families-of-quasi-parabolic-bundles}
		A \textit{family of quasi-parabolic vector bundles of type $(\boldsymbol{e,m})$ 
			on $C$ parametrized by $S$} is a vector bundle $\mc{F}$ on $C\times S$
		together with, for each $p\in D$, a section $s_p$ of
		$\mathrm{Flag}_{\underline{m}(p)}\bigl(\mc{F}|_{\{p\}\times S}\bigr)
		\longrightarrow S$.
	\end{definition}
	
	\noindent
	The moduli stack of quasi-parabolic vector bundles on $C$ of rank $n$,
	determinant $L\in\Pic(C)$ and quasi-parabolic data $(\boldsymbol{e,m})$ is
	denoted $\mc{PM}_C^{\boldsymbol{e,m}}(n,L)$. Its objects over $S$ are
	families of quasi-parabolic vector bundles of type $(\boldsymbol{e,m})$ on $C$ parametrized by $S$ as in Definition~\ref{def:families-of-quasi-parabolic-bundles} with
	$\det(\mc{F})\simeq q^*L$ where $q: C\x S\to C$ is the projection morphism.
	
	\noindent
	Let $D$ denote the set of parabolic points. Fix a quasi-parabolic data $(\boldsymbol{e,m})$ of rank $n$ along $D$. Let us denote by \begin{align*}\label{eqn:weights}
		\boldsymbol{\alpha} : =\{(\alpha_{_{p,1}}<\alpha_{_{p,2}}<\cdots<\alpha_{_{p,e_p}})\ \mid\ p\in D\} 
	\end{align*}
	a set of \textit{weights} assigned to this data, where $\alpha_{p,i}$'s are real numbers lying between $0$ and $1$.
	
	\begin{definition}\label{def:parabolic-bundles}
		Let $E$ be a quasi-parabolic vector bundle of rank $n$ on $C$ associated to the data $(\boldsymbol{e,m})$. 
		\begin{enumerate}[(1)]
			\item The $\boldsymbol{\alpha}$-\textit{degree} of $E$ is defined as the real number
			\[\deg_{\boldsymbol{\alpha}}(E) := \deg(E)+\sum_{p\in D}\sum_{i=1}^{e_p}m_{_{p,i}}\alpha_{_{p,i}}.\]
			\item The $\boldsymbol{\alpha}$-\textit{slope} of $E$ is defined to be the real number
			\[\mu_{\boldsymbol{\alpha}}(E):=\deg_{\boldsymbol{\alpha}}(E)/n.\] 
		\end{enumerate}
	\end{definition}
	This allows us to define the notion of $\boldsymbol{\alpha}$-semistability and $\boldsymbol{\alpha}$-stability for a quasi-parabolic vector bundle as usual. The quasi-parabolic moduli stack
	$\mc{PM}_C^{\boldsymbol{e,m}}\left(n,L\right)$ contains an open substack consisting of $\boldsymbol{\alpha}$-semistable quasi-parabolic vector bundles on $C$ of rank $n$, determinant $L$ and quasi-parabolic data $(\boldsymbol{e,m})$ along $D$. We shall denote this open substack as
	\begin{align*}
		\mc{PM}^{\boldsymbol{\alpha}-ss,\boldsymbol{e,m}}_C\left(n,L\right).
	\end{align*}
	
	\noindent
	We fix a finite set $D \subset C$ as above, and a quasi-parabolic data $(\boldsymbol{e,m})$ along $D$. Then one can construct a tame stacky curve $\mf{C}$ with coarse space $C$ (as an iterated root stack over $C$ along $D$), such that if $\underline{m}$ denotes the set of multiplicities at the stacky points (see Definition \ref{def:type of a vector bundle on stacky curve}), then one has an identification of moduli stacks (see \cite[Th\'eor\`eme~4]{Bor07}, \cite[Lemma 7.9]{Nir08} and \cite[Theorem 3.2.5]{Taams})
	$$\mc{PM}_C^{\boldsymbol{e,m}}(n,L) \cong \mc{M}_{\mf{C}}\bigl(n,\underline{m},L\bigr)\,.$$
	\noindent 
	To construct $\mf C$ as above, let $D=\{p\}$ for simplicity. The quasi-parabolic data consists of a number $e_p$ and a tuple $\boldsymbol{m'}:=(m_0,m_1,\ldots, m_{e_p-1})$ satisfying $\sum_{i=0}^{e_p-1}m_i=n$. Define 
	\[e'_p:= e_p\,\,\text{if}\,\,\mathrm{char}(k) \nmid e_p,\,\,\,e'_p:= e_p+1\,\,\mathrm{otherwise}\,.\]
	Consider the root stack $\mf{C}:= \sqrt[^{e'_p}]{p/C}$. Since $e'_p$ is coprime to $\mathrm{char}(k)$, the root stack $\mf C$ is smooth (see proof of~\cite[Proposition 2.4]{BL23}), Deligne-Mumford (see~\cite[Theorem 2.3.3]{Cad07}) and tame \ie a tame stacky curve. Let $\pi:\mf C \to C$ denote the natural map. Let $p'\in \mf C$ denote the unique point mapping to $p$ under $\pi$. Note that $p'$ is a stacky point of order $e'_p$. Also, define the multiplicity vector $\underline{m}$ at the stacky point $p'$ (see Definition \ref{def:type of a vector bundle on stacky curve}) by the following tuple: $$\underline{m}:= (m_0,m_1,\ldots, m_{e_p-1})\,\,\text{if}\,\, e'_p=e_p,\,\,\,\text{and}\,\,\,\underline{m}:= (m_0,m_1,\ldots, m_{e_p-1},0)\,\,\text{if}\,\,e'_p=e_p+1\,.$$ 
	Given a vector bundle $\mc{F}$ on $\mf C$, consider the vector bundle $\pi_*\mc F$ on $C$ equipped with the filtration
	\begin{align}\label{eq:par-filtration}
	    \pi_*\mc F\supseteq \pi_*\bigl(\mc F\otimes \mc{O}_{\mf C}(-\mc G_{p'})\bigr)\supseteq \pi_*\bigl(\mc F\otimes \mc{O}_{\mf C}(-2\mc G_{p'})\bigr)\supseteq\cdots\supseteq \pi_*\bigl(\mc F\otimes \mc{O}_{\mf C}(-e'_p\mc G_{p'})\bigr)=\pi_*\mc{F}(-p)\,.
	\end{align}
	Note that, by construction, the multiplicity vector at $p$ is just $\boldsymbol{m}:=\underline{m}$. This gives rise to a quasi-parabolic vector bundle on $C$ with quasi-parabolic data $(\boldsymbol{m},e'_p)$ along $D=\{p\}$. It follows that the moduli stacks
	\[\mc{PM}_C^{e_p,\boldsymbol{m'}}(n,L)\cong \mc{PM}_C^{e'_p,\boldsymbol{m}}(n,L) \cong \mc{M}_{\mf{C}}\bigl(n,\underline{m},L\bigr)\] are isomorphic. 
    
	 Furthermore, given a set of weights $\boldsymbol{\alpha}$ as above, one can construct a generating sheaf $\mathcal{E}_{\boldsymbol{\alpha}}$ on $\mf C$ such that the above isomorphism of stacks  restricts to an isomorphism on the open substacks of semistable points:
	\begin{equation}\label{eq:parabolic-stacky}
		\mc{PM}^{\boldsymbol{\alpha}-ss,\boldsymbol{e,m}}_C\left(n,L\right)\cong \mc{M}_{\mf{C}}^{\mathcal{E}_{\boldsymbol{\alpha}}-\text{ss}}\bigl(n,\underline{m},L\bigr)\,.
	\end{equation} 
	\noindent
	In \cite[Theorem 4.3]{CC25}, it was shown that, the
	moduli stack $\mc{PM}^{\boldsymbol{\alpha}\text{-}ss,\boldsymbol{e,m}}_C(n,L)$
	is $\bb{A}^1$--connected under the assumptions that the weights $\boldsymbol{\alpha}$ are small and generic, and also $\gcd(n,\deg L)=~1$.  The identification \eqref{eq:parabolic-stacky} allows us to extend the result for any $\boldsymbol{\alpha}$, $n$ and $L$:
	\begin{corollary}\label{cor:parabolic} Let $C$ be a  geometrically connected smooth projective curve over an infinite field $k$ and $D$ be a finite set of closed points with residue fields separable over $k$.  Let $\alpha$ be a set of weights on $C$.  Then the moduli stack 
		$$\mc{PM}^{\boldsymbol{\alpha}-ss,\boldsymbol{e,m}}_C\left(n,L\right)$$ of $\boldsymbol{\alpha}$-semistable parabolic vector bundles on $C$ of rank $n$,
		determinant $L\in\Pic(C)$ and quasi-parabolic data $(\boldsymbol{e,m})$ is $\#A^1$-connected if and only if it admits a $k$-rational point. 
		
	\end{corollary}
	
	In fact, by Corollary~\ref{thm:-genusatleast2-A1sc} and the identification \eqref{eq:parabolic-stacky}, whenever $C$ is of genus $\ge 2$ and $C(k)\neq \emptyset$, the moduli stack $\mc{PM}^{\boldsymbol{\alpha}-ss,\boldsymbol{e,m}}_C\left(n,L\right)$ admits a $k$-rational point and hence is $\#A^1$-connected.

	\section{\texorpdfstring{Existence of rational points on $\mc{M}^{\mc{E}-ss}_{\mf{C}}\bigl(n,\underline{m},L\bigr)$}{}}\label{sec:rpt}
	In  this section, we present sufficient conditions so that the moduli stack $\mc{M}_{\mf{C}}\bigl(n,\underline{m},L\bigr)$ and its open substack $\mc{M}^{\mc{E}-ss}_{\mf{C}}\bigl(n,\underline{m},L\bigr)$ admit a $k$-rational point.  
	Recall we have the natural map 
	\[\pi_*: \mc{M}_{\mf{C}}\bigl(n,\underline{m},L\bigr)\longrightarrow \mc{M}_C(n,L)\]
	induced by the pushforward $\pi_*$ under the coarse space map $\pi:\mf{C}\to C$. For our next result we need to recall the description of the fibers of $\pi_*$ over a $k$-point $y$, given by a vector bundle $E$ on $C$
	\begin{center}
		\begin{tikzcd}
			{F_y} && {\mathcal{M}_{\mathfrak{C}}\bigl(n,\underline{m},L\bigr)} \\
			\\
			{\textrm{Spec }k} && {\mathcal{M}_C(n,L)}
			\arrow[from=1-1, to=1-3]
			\arrow[from=1-1, to=3-1]
			\arrow["{\pi_*}", from=1-3, to=3-3]
			\arrow["y"', from=3-1, to=3-3]
		\end{tikzcd}
	\end{center}
	
	For simplicity, let us consider the case of a single stacky point $p$ with stabilizer $\mu_{e_p}$ and residue field $k(p)$. Using the identification in \textsection \ref{sec:parabolic}
	
	\begin{center}
		\begin{tikzcd}
			{\mathcal{PM}_C^{\boldsymbol{e,m}}(n,L) } && { \mathcal{M}_{\mathfrak{C}}\bigl(n,\underline{m},L\bigr)} \\
			\\
			& {\mathcal{M}_C(n,L)}
			\arrow["\simeq", from=1-1, to=1-3]
			\arrow[from=1-1, to=3-2]
			\arrow["{\pi_*}", from=1-3, to=3-2]
		\end{tikzcd}
	\end{center}
	the functor of points of $F_y$ can be identified with the sheaf of sets
	$$S \mapsto \{\text{flags of type }\underline d(p)\text{ in }(E_S)_p\}$$
	
	where $S$ is any $k$-scheme, $(E_S)_p:=E_S|_{\{p\}\times S}$ and $\underline d(p)$ is the type of the flag uniquely determined by $\underline{m}$ obtained  via \eqref{eq:par-filtration}. Note that $\{p\}\times S$ is defined over $k(p)$.
	
	Given a finite extension $L/k$ and a scheme $X$ over $L$, the Weil restriction of $X$, denoted $\textrm{Res}_{L/k}X$ is the $k$-scheme, if it exists, given by the functor of points 
	$$S \mapsto \textrm{Res}_{L/k}X(S) := X(S \times_k L)$$
	for any $k$-scheme $S$. It is well-known that $\textrm{Res}_{L/k}$ takes quasi-projective varieties to quasi-projective varieties. For more details on the properties of Weil restriction, see \cite[Section 7.6]{BLR90}.
	
	By definition of Weil restriction and the description of the functor of points of $F_y$ above, one concludes that the fiber $F_y$ is represented by the Weil restriction $$\textrm{Res}_{k(p)/k}\textrm{Flag}_{\underline d(p)}(E_p)$$
	where $\textrm{Flag}_{\underline d(p)}(E_p)$ is the partial flag variety over $k(p)$ parametrizing the flags of type $\underline d(p)$ in the $k(p)$-vector space $E_p$. In the general case of multiple stacky points, the fiber $F_y$ is represented by the product
	$$\prod_{p}\textrm{Res}_{k(p)/k}\textrm{Flag}_{\underline d(p)}(E_p).$$
	By \cite[Remark 1.1.34]{Taams}, $k(p)$ is a separable extension of $k$ for each stacky point $p$. Thus, by \cite[\textsection 7.6, Proposition 5]{BLR90}, $\textrm{Res}_{k(p)/k}\textrm{Flag}_{\underline d(p)}(E_p)$ is proper and smooth over $k$. This shows that $\pi_*$ is a smooth proper morphism.
	
	Finally we note that the Weil restriction functor $\textrm{Res}_{k(p)/k}$ takes open immersions to open immersions. More precisely, by \cite[Section 7.6, Proposition 2(i)]{BLR90}, we have the following: Let $L/k$ be a finite extension and $X$ be a quasiprojective $L$-scheme. For an open immersion $U \to X$, the induced morphism $\textrm{Res}_{L/k}(U) \to \textrm{Res}_{L/k}(X)$ is an open immersion. As $\textrm{Res}_{L/k}(\mathbb{A}^n_L) = \mathbb{A}^{nd}_k$, where $d = [L:k]$, we get that $\textrm{Res}_{k(p)/k}\textrm{Flag}_{\underline d(p)}(E_p)$ are $k$-rational for all $p$, since by the theory of Bruhat cells, the partial flag varieties $\textrm{Flag}_{\underline d(p)}(E_p)$ are $k(p)$-rational.
	
	\begin{lemma} Let $\mf{C}$ be a projective stacky curve over an infinite field $k$ and $\pi:\fr C\to C$ be the coarse space map. Assume that the curve $C$ is of genus at least $2$ and $C(k)\neq \emptyset$. Then
		any non-empty open substack $\mc{U}\subset \mc{M}_{\mf{C}}\bigl(n,\underline{m},L\bigr)$ contains a $k$-rational point. 
	\end{lemma}\begin{proof}
		We shall show that any nonempty open substack $\mc{U}\subset \mc{M}_{\mf{C}}\bigl(n,\underline{m},L\bigr)$ contains a $k$-rational point whenever the genus of $C$ is $\ge 2$.  Since the map of stacks \[\pi_*: \mc{M}_{\mf{C}}\bigl(n,\underline{m},L\bigr)\longrightarrow \mc{M}_C(n,L)\] is smooth (in particular, flat and locally of finite presentation), it is an open morphism. It follows that $\mc{U}':= \pi_*(\mc{U})$ is a non-empty open substack of $\mc{M}_C(n,L)$. Since genus of $C$ is at least $2$ and $k$ is an infinite field, by \cite[Lemma 1.5]{Hof07} the open substack $\mc{U}'$ admits a $k$-rational point, say $y: \Spec k \to \mc{U}'$. Consider the composition 
		$$\tilde{y}: \Spec k \to \mc{U}' \hookrightarrow \mc{M}_C(n,L).$$
		By the preceding discussion, the fiber $F_{\tilde{y}}:= \Spec k\times_{\~y,~\mc{U}'} \mc{M}_{C}\bigl(n,L\bigr)$ is represented by the $k$-scheme  $$\prod_{p}\textrm{Res}_{k(p)/k}\textrm{Flag}_{\underline d(p)}(E_p)$$ which is a geometrically irreducible
		$k$-rational variety. As $y$ lies in the image of $\mc{U}$ under $\pi_*$, the open subscheme $U:= F_y\cap \mc{U}\neq \emptyset$. Since $F_y$ is rational over $k$ and $k$ is infinite, $F_y(k)$ is dense in $F_y$. Hence $U$, and consequently $\mc{U}$, also contains a $k$-rational point.  
	\end{proof}
	
	In the case, the genus of $C$ is at least $2$, we can make the following uniform assertion. 
	\begin{corollary}\label{thm:-genusatleast2-A1sc}
		\label{prop:genus-atleast2} Let $\mf{C}$ be a projective stacky curve over an infinite field $k$ and $\pi:\fr C\to C$ be the coarse space map. If the curve $C$ is of genus  $\geq 2$, and $C(k)\neq \emptyset$, then the moduli stack  $\mc{M}^{\mc{E}-\text{ss}}_{\mf{C}}\bigl(n,\underline{m},L\bigr)$ admits a $k$-rational point. Consequently, it is $\bb{A}^1$-connected.
		
	\end{corollary}
	\begin{proof}
		By \cite[Section 5]{MS80} and the identification \eqref{eq:parabolic-stacky}, the open substack $\mc{M}^{\mc{E}-\text{ss}}_{\mf{C}}\bigl(n,\underline{m},L\bigr)$ admits a $\-k$-point, hence non-empty. It follows from the above lemma that it admits a $k$-rational point. Hence the stack is $\bb{A}^1$-connected by~\Cref{thm:main-A1sc}.
	\end{proof}
	\begin{remark}
		When the genus of $C$ is 0, the moduli stack $\mc{M}^{\mc{E}-\text{ss}}_{\mf{C}}\bigl(n,\underline{m},L\bigr)$ may be empty. This follows from the identification made in \eqref{eq:parabolic-stacky}, together with the fact that  the moduli stack of parabolic semistable vector bundles on $C$ can be empty depending on the parabolic weights (see~\cite[Section 5]{BH95}, \cite[Lemma 5.2, Theorem 5.3]{AW98}). 
	\end{remark}

\end{document}